\documentclass[11pt, reqno]{amsart}
\usepackage{amsmath, amsthm, amscd, amsfonts, amssymb, graphicx, color}
\usepackage[bookmarksnumbered, colorlinks, plainpages]{hyperref}
\hypersetup{colorlinks=true,linkcolor=red, anchorcolor=green, citecolor=cyan, urlcolor=red, filecolor=magenta, pdftoolbar=true}
\newtheorem{theorem}{Theorem}[section]

\newtheorem{corollary}[theorem]{Corollary}
\theoremstyle{definition}
\newtheorem{definition}[theorem]{Definition}
\newtheorem{example}[theorem]{Example}

\theoremstyle{remark}

\numberwithin{equation}{section}
\allowdisplaybreaks
\begin{document}

\setcounter{page}{1}

\title[Convolution Properties of Orlicz Spaces]{Convolution Properties of Orlicz Spaces on hypergroups}

\author[A.R. Bagheri Salec, Vishvesh Kumar and S.M. Tabatabaie]{Ali Reza Bagheri Salec, Vishvesh Kumar and Seyyed  Mohammad Tabatabaie}

\address{Department of Mathematics, University of Qom, Qom, Iran.}
\email{\textcolor[rgb]{0.00,0.00,0.84}{r-bagheri@qom.ac.ir}}

\address{Department of Mathematics: Analysis, Logic and Discrete Mathematics,
	Ghent University, Belgium}
\email{\textcolor[rgb]{0.00,0.00,0.84}{vishveshmishra@gmail.com}}

\address{Department of Mathematics, University of Qom, Qom, Iran.}
\email{\textcolor[rgb]{0.00,0.00,0.84}{sm.tabatabaie@qom.ac.ir}}

\address{
\newline
}
\dedicatory{Dedicated to Prof. Kenneth A. Ross on his 85th birthday }


\begin{abstract}
In this paper, for a locally compact commutative hypergroup $K$ and for a pair $(\Phi_1, \Phi_2)$ of Young functions satisfying sequence condition, we give a necessary condition in terms of aperiodic elements of the center of $K,$  for the convolution $f\ast g$ to exist a.e., where $f$ and $g$ are arbitrary elements of Orlicz spaces $L^{\Phi_1}(K)$ and $L^{\Phi_2}(K)$, respectively. As an application, we present some equivalent conditions for compactness of a compactly generated locally compact abelian group. Moreover, we also characterize compact convolution operators from $L^1_w(K)$ into $L^\Phi_w(K)$ for a weight $w$ on a locally compact  hypergroup $K$. 
\end{abstract} \maketitle

\section{Introduction}
The well-known $L^p$-conjecture stated that if $1<p<\infty$ and $G$ is a locally compact group, then the  Lebesgue space $L^p(G)$ is a Banach algebra under the convolution product if and only if $G$ is compact. Saeki \cite{Saeki} settled this conjecture by  quite elementary means, much more elementary than some of the
proofs of earlier partial results \cite{Raja, Rajaze}. In \cite{abt}, the authors mentioned that for each $p>2$, if $f\ast g$ exists a.e. for all $f,g\in L^p(G)$, then $G$ is compact, and so automatically $f\ast g\in L^p(G)$; see also \cite{que}. In the setting of hypergroups, this result was studied in  \cite{tab1} under some conditions. 
H. Hudzik, A. Kami\'nska and J. Musielak in \cite[Theorem 2]{hud} presented some equivalent conditions for an Orlicz space $L^\Phi(G)$ to be a convolution Banach algebra:
\begin{theorem}\label{orl}
	If $G$ is a locally compact abelian group and $\Phi$ is a Young function satisfying $\Delta_2$-condition, then the following conditions are equivalent:
	\begin{enumerate}
		\item $L^\Phi(G)$ is a Banach algebra under convolution;
		\item $L^\Phi(G)\subseteq L^1(G)$;
		\item $\lim_{x\rightarrow 0^+}\frac{\Phi(x)}{x}>0$ or $G$ is compact.
	\end{enumerate}
\end{theorem}
Recently, A. Osan\c{c}l{\i}ol and S. \"Oztop in \cite{ozt} studied the weighted Orlicz algebras on locally compact groups (see also \cite{tab3}). They proved that, even for a non-compact group $G$, if $L^\Phi_w(G)\subseteq L^1_w(G)$ for a weight $w$, then $L^\Phi_w(G)$ is a convolution Banach algebra. In \cite{ozt2}, these results were extended to the hypergroup case (see \cite{kum2} for unweighted case).
In \cite{tab1}, for a compactly generated abelian group $G,$ it is proved that if $\Phi$ is a Young function with $\Delta_2$-condition and satisfying a sequence condition, then $L^\Phi(G)$ is a convolution Banach algebra if and only if $f\ast g$ exists a. e. for all  $f,g\in L^\Phi(G)$.

A main motivation for writing this paper is the above background and the following result from \cite[Corollary 1.4]{que} about Lebesgue spaces on locally compact groups.
\begin{theorem}\label{111}
Let $G$ be a locally compact abelian group and $1<p,q<\infty$. Then, $L^p(G)\ast L^q(G)\subseteq L^p(G)$ if and only if  $G$ is compact.
\end{theorem}

In Section 3, we  give a version of this result for Orlicz spaces on locally compact hypergroups.  Indeed,  for a pair  $(\Phi_1,\Phi_2)$ of Young functions satisfying the sequence condition \eqref{seq} (see Definition \ref{defseq}), we show that if for each $f\in L^{\Phi_1}(K)$ and $g\in L^{\Phi_2}(K)$,  $f\ast g$ exists almost everywhere, then there is no aperiodic element in $Z(K)$ with respect to the action $Z(K)\curvearrowright (K,\lambda)$, where  $K$ is a locally compact commutative hypergroup equipped with an invariant measure $\lambda$ and $Z(K)$ is the center of $K$. As an application, among other results, we prove that a compactly generated abelian group $G$ is compact if and only if for each pair $(\Phi_1,\Phi_2)$ of Young functions satisfying the sequence condition \eqref{seq} and for each $f\in L^{\Phi_1}(G)$ and $g\in L^{\Phi_2}(G)$, $f\ast g$ exists a.e. We note that if we consider Lebesgue spaces $L^{p_1}(G)$ and $L^{p_2}(G)$, where $p_1,p_2>2$, the novel conclusion Corollary \ref{cor2} is obtained for Lebesgue spaces too, because the Young functions $\Phi_{p_i}, i=1,2$  defined by $\Phi_{p_i}(x):=|x|^{p_i}$ satisfy the sequence condition \eqref{seq}. 

In section 4,  we fix a function $g\in C_c(K)$, and study the compact convolution operator $f\mapsto f\ast g$ from $L^1_w(K)$ into $L^\Phi_w(K)$, where $K$ is a locally compact hypergroup and $w$ is a weight function on $K$. We show that this operator is compact if and only if the function $x\mapsto\frac{1}{w(x)}\|{\rm L}_xg\|_{\Phi,w}$ on $K$ vanishes at infinity. This conclusion is an Orlicz space version of the main result of  \cite{ghamed}. It is also a generalization of one result in \cite{tab3} on locally compact hypergroups. It is worth noting that the convolution operators in different Orlicz spaces were also studied  before by O'Neil \cite{Oneil} and Kami\'nska and Musielak \cite{KaminskaMusi}.

In the next section, we recall some definitions and notation concerning locally compact hypergroups and Orlicz spaces, and also state some facts about aperiodic elements of a group action on a measure space which are used in the proof of our main result.

\section{Preliminaries}
\subsection{Locally Comapct Hypergroups}
Locally compact hypergroups were introduced in the papers \cite{dun, jew, spe} in the 1970's. The main references for us on this structure are the paper \cite{jew} (in which hypergroups are called \emph{convos}) and the monograph \cite{blm}.
Let $K$ be a locally compact Hausdorff space. We denote the space of all bounded (complex-valued) Radon measures on $K$ by $\mathcal M(K)$ and all those in $\mathcal M(K)$ which are non-negetive by $\mathcal{M}^+(K)$. The support of each measure $\mu\in\mathcal M(K)$ is denoted by ${\rm supp}\mu$, and for each $x\in K$, $\delta_x$ denotes the Dirac measure at $x$. The space $K$ is called a \emph{locally compact hypergroup} (or simply a \emph{hypergroup}) if there are a \emph{convolution} $*:\mathcal M(K)\times \mathcal M(K)\rightarrow \mathcal M(K)$, an \emph{involution} $x\mapsto x^-$ on $K$, and an element $e$ (called the \emph{identity} element) such that the following conditions hold:
\begin{enumerate}
	\item $(\mathcal M(K), +, \ast)$ is a complex Banach algebra;
	\item for all nonnegative measures $ \mu,\nu\in \mathcal M(K)$, $\mu*\nu$ is also a nonnegative measure in $\mathcal M(K)$ and  the mapping
	$(\mu,\nu) \mapsto \mu * \nu $  from  $\mathcal M^+(K) \times \mathcal M^+ (K) $ to $\mathcal M^+ (K)$ is continuous,
	where $\mathcal M^+ (K)$ is equipped with the cone topology;
	\item for all $x,y \in K$,  $\delta_x * \delta_y$ is a probability measure with compact support;
	\item the mapping $(x,y) \mapsto\text{supp}(\delta_x * \delta_y)$ from $K \times K$ into the space $\mathfrak{C}(K)$
	of all non-empty compact subsets of $K$ is continuous, where $\mathfrak{C}(K)$ is equipped with Michael topology   whose subbasis is the family of all
	$\mathfrak{C}_{U,V}:=\{A \in \mathfrak{C} (K): A \cap U \neq \varnothing\text{ and } A \subseteq V\}$, where $U$  and  $V$ are open subsets of $K$;
	\item  for each $ x\in K$,     $\delta_e * \delta_x =\delta_x= \delta_x * \delta_e$;
	\item the mapping $x\mapsto x^-$ on $K$ is a homeomorphism, and for each $x,y \in K$ we have $ (x^{-})^{-}=x $ and  $ (\delta_x * \delta_y)^{-}=\delta_{y^{-}}*\delta_{x^{-}} $, where for each $\mu\in \mathcal M(K)$ and Borel set $E\subseteq K$, $\mu^-(E):=\mu\left(\{x^-:\,x\in E\}\right)$. Also,  $ e\in {\rm supp}(\delta_{x}*\delta_{y}) $ if and only if $ x=y^{-} $. 
\end{enumerate}

 A hypergroup $K$ is called \emph{commutative} if for each $x,y\in K$, $\delta_x\ast\delta_y=\delta_y\ast\delta_x$. Any locally compact group is a hypergroup. 
 See the book and papers \cite{blm, dun2, voi, KKA, KKAadd,KumarRamsey} for more examples including double coset spaces, polynomial hypergroups and orbit hypergroups and their applications.

Throughout, we assume that $(K,\ast,\cdot^-,e)$ is a locally compact hypergroup with a left-invariant measure, i.e. a non-negative Radon measure $\lambda$ on $K$ such that for each $x\in K$, $\delta_x\ast \lambda=\lambda$. It is known that any commutative hypergroup, compact hypergroup, discrete hypergroup and amenable hypergroup admits such a measure \cite{blm}. 

For each element $x\in K$ and  Borel subsets $E,F$ of $K$ we denote
$$
x\ast F:=\bigcup_{y\in F}{\rm supp}(\delta_x\ast\delta_y),\quad E\ast F:=\bigcup_{t\in E}\left(t\ast F\right).
$$

The \emph{convolution} of any two measurable functions $f,g:K\rightarrow\mathbb C$ is defined by
$$(f\ast g)(x):=\int_K f(y)g(y^-\ast x)\,d\lambda(y),\quad (x\in K),$$
when this integral exists, where
$$g(y^-\ast x):=\int_K g(t)\,d(\delta_{y^-}\ast\delta_x)(t).$$

If $\mu\in \mathcal M(K)$ and $f$ is a Borel measurable function on $K$, the convolution $\mu*f$ is defined by:
$$(\mu*f)(x) = \int_K f(y^-*x)\ d\mu(y),\quad(x\in K).$$
In particular, $(\delta_{z^-}*f)(x)=f_z(x)=({\rm L}_z f)(x)$ for $x, z \in K.$ 
\subsection{Group Action on a Measure space}
\begin{definition}\label{ac}
	Let $G$ be a locally compact group, $X$ be a locally compact Hausdorff space, and $\mu$ be a non-negative Radon measure on $X$. We say that a continuous function
	$$G\times X\longrightarrow X,\quad (s,x)\mapsto sx,\quad (s\in G, x\in X)$$
	is an \emph{action} of $G$ on the measure space $(X,\mu)$ (and write $G\curvearrowright (X,\mu)$) if
	\begin{enumerate}
		\item [{\rm(i)}] for each $x\in X$, $ex=x$, where $e$ is the identity element of $G$;
		\item [{\rm(ii)}]for each $s,t\in G$ and $x\in X$, $s(tx)=(st)x$;
		\item [{\rm(iii)}]the measure $\mu$ is $G$-invariant, i.e., for each $s\in G$ and any Borel subset $E$ of $X$, $sE:=\{sx:\,x\in E\}$ is also a Borel subset of $X$ and $\mu(sE)=\mu(E)$.
	\end{enumerate}
\end{definition}
Assume that $H$ is a closed subgroup of a locally compact group $G$ with modular functions $\Delta_H: H \rightarrow (0, \infty)$ and $\Delta_G:G \rightarrow (0,\infty)$. Denote the restriction of $\Delta_G$ on $H$ by $\Delta_G|_{H}.$ If $\Delta_H=\Delta_G|_{H}$  then there exists a $G$-invariant Radon measure $\mu$ on $G/H$, and $G$ naturally acts on the quotient space $(G/H,\mu)$. For more study on this topic we refer to the book \cite[Chapter IV]{ker}.
\begin{definition}
	An element  $a$ of a locally compact group $G$ is called \emph{compact} if the closed subgroup $G(a)$ generated by $a$ is compact.
\end{definition}

For some details about compact elements of $G$ see \cite{hew}. In the literature,  non-compact elements of a group $G$ are also called \emph{aperiodic elements} (for example see \cite{che, che1}). Trivially, in any discrete group, an element is aperiodic if and only if it has infinite order (i.e. it is not a torsion element). Recently, these elements have been used to study linear dynamical properties of weighted translation operators on locally compact groups; see \cite{chenn, che1}. It is known that any element, except the identity, of the non-discrete additive group $\mathbb R^n$, the Heisenberg group, and the affine group is aperiodic. By \cite[Lemma 2.1]{che}, if $G$ is a second countable group, an element $a\in G$ is aperiodic if and only if  for each compact subset $E$ of $G$, there is an integer $N>0$ such that for each $n\geq N$, $E\cap a^nE=\varnothing$. This equivalence leads one to give the following definition which recently has been used to present a sufficient and necessary condition for a weighted translation, generated by a group action, to be disjoint topologically transitive.
\begin{definition}\label{ap}
	 An element $a\in G$ is called \emph{aperiodic} with respect to a given action $G\curvearrowright (X,\mu)$ if for each compact subset $E\subseteq X$, there exists an integer  $N>0$ such that for each $n\geq N$, $E\cap a^n E=\varnothing$.
\end{definition}
Thanks to \cite[Lemma 2.1]{che}, the aperiodic elements of a second countable group $G$ are same as the aperiodic elements of $G$ with respect to the natural action of $G$ on itself. In this paper, we will apply this concept regarding the action of the center of a hypergroup on the whole of hypergroup.
 The \emph{center} $Z(K)$ of a commutative hypergroup $K$ is defined as the set of all  $x\in K$ such that for each $y\in K$, $\text{supp}(\delta_x\ast\delta_y)$ is a singleton. In other words, for each $x\in Z(K)$ and $y\in K$, there is an element $\alpha(x,y)\in K$ such that $\delta_x\ast\delta_y=\delta_{\alpha(x,y)}$. The center $Z(K)$ is the maximal subgroup of $K$, and naturally acts on $(K,\lambda)$ by the mapping $(x,y)\mapsto \alpha(x,y)$ \cite[Section 10.4]{jew}.  In the sequel, we denote $x\ast y:=\alpha(x,y)$ for all $x\in Z(K)$ and $y\in K$. Also, for each $x\in Z(K)$ and $n\in\mathbb N$ we put $x^n:=x\ast\ldots\ast x$ ($n$ times), and $x^{-n}:=(x^-)^n$. For more details and examples about center of hypergroups see \cite{ross}.  We use this action in the main result of the paper. One can easily see that for each element $x\in Z(K)$ and Borel subset $E$ of $K$, 
 \begin{equation}\label{inv}
 \lambda(x\ast E)=\lambda(E)
 \end{equation}
  while this equality does not hold for arbitrary elements of $K$; see \cite{jew}.
\subsection{Orlicz Spaces}
In this subsection, we recall some basic definitions and notation about Orlicz spaces; see the monographs \cite{rao,rao1} on this topic.
	A convex even mapping $\Phi:\mathbb{R}\rightarrow[0,\infty)$ is called a \emph{Young function} if $\Phi(0)
	=0$ and $\lim_{t\rightarrow\infty}\Phi(t)=\infty$.
	The \emph{complementary} of a Young function $\Phi$ is defined by
	$$\Psi(t):=\sup\{s|t|-\Phi(s):\,s\geq 0\},\quad(t\in\mathbb{R}).$$

In the sequel, $\Psi$ always denotes the complementary of a given Young function $\Phi$. 
The set of all Borel measurable functions $f:K\rightarrow\mathbb C$ such that for some $\alpha>0$,
$$\int_K\Phi(\alpha|f(x)|)\,d\lambda(x)<\infty,$$
is denoted by $L^\Phi(K)$. We assume that two elements of $L^\Phi(K)$ are the same if they are equal $\lambda$-a.e. For each $f\in L^\Phi(K)$ we define 
$$\|f\|_{\Phi}:=\sup\left\{\int_K |fv|\,d\lambda:\,\int_K\Psi(|v(x)|)\,d\lambda(x)\leq 1\right\}.$$

The pair $(L^\Phi(K),\|\cdot\|_\Phi)$ is called an \emph{Orlicz space}. Since $\lambda$ is a regular measure on $K$, by \cite[Chapter III, Proposition 11]{rao},  $(L^\Phi(K),\|\cdot\|_\Phi)$  is a Banach space.

 A Young function $\Phi$ is said to be
	\emph{$\Delta_2$-regular} (denoted by $\Phi\in \Delta_2$), if there exists a constant $C>0$  such that
	$\Phi(2t)\leq C\Phi(t)$ for all $t\geq 0,$ provided that $\lambda(K)=\infty$. If $\lambda(K)<\infty,$ we say $\Phi\in \Delta_2,$ if there exist two constants $C>0$ and $t_0 \geq 0$  such that
	$\Phi(2t)\leq C\Phi(t)$ for all $t\geq t_0.$  At times, we also say that $\Phi$ satisfies $\Delta_2$-condition if $\Phi \in \Delta_2.$
	Denote the space of all complex-valued compactly supported continuous functions on $K$ by $C_c(K).$ Then it is known that  
	for $\Phi\in \Delta_2,$ the space $C_c(K)$ is dense in $L^{\Phi}(K)$ (see \cite [Proposition 11, page 18]{rao1}).

Orlicz spaces are a more applicable generalization of Lebesgue spaces. In fact, for each $1<p<\infty$, the function $\Phi_p$ defined by  $\Phi_p(t):=|t|^p$ for all $t\in\mathbb{R}$, is a Young function and the Orlicz space $L^\Phi(K)$ is same as the usual Lebesgue space $L^p(K,\lambda)$. Orlicz spaces have been studied for several decades; see  \cite{kum2,KR,Kumar1,tab4, KT} for some interesting recent developments related to Orlicz spaces on locally compact hypergroups.

\section{Convolution of Two Orlicz Spaces}
In this section, we study the convolution properties of two different Orlicz spaces on locally compact hypergroups in terms of aperiodic elements of their centers. We will also derive interesting results for locally compact groups.
Before stating the main result of the section, we need to  introduce a class of Young functions.
\begin{definition} \label{defseq}
	Let $\Phi_1$ and $\Phi_2$ be  Young functions. We say that the pair $(\Phi_1,\Phi_2)$ satisfies the \emph{sequence condition} if there are two sequences $(\alpha_n)$ and $(\beta_n)$ of nonnegative numbers such that 
	\begin{equation}\label{seq}
	\sum_{n=1}^\infty\Phi_1(\alpha_n)<\infty,\quad\sum_{n=1}^\infty\Phi_2(\beta_n)<\infty\quad\text{and}\quad \sum_{n=1}^\infty\alpha_n\beta_n=\infty.
	\end{equation}
\end{definition}
\begin{example}			
			For each $p\geq 1$ and $\gamma\geq 0$ define the function $\Phi_{p,\gamma}$ by $\Phi_{p,\gamma}(x):=|x|^p\left(\ln(1+|x|)\right)^\gamma$ for all $x\in\mathbb{R}$. We denote the set of all $(p,\gamma)$ such that $p+\gamma>2$ and $\Phi_{p,\gamma}$ is  a Young function by $\Omega$. Then, for each $(p_1,\gamma_1),(p_2,\gamma_2)\in \Omega$, setting $\alpha_n=\beta_n:=\frac{1}{\sqrt{n}}$, one can see that $(\Phi_{p_1,\gamma_1},\Phi_{p_2,\gamma_2})$ satisfies the sequence condition (\ref{seq}).
\end{example}
Now, we are ready to present one of the main results of this paper.
\begin{theorem}\label{main}
Let $K$ be a locally compact commutative hypergroup\footnote{So $K$ has  an invariant measure $\lambda$.}. Suppose that $\Phi_1$ and $\Phi_2$ are two Young functions such that the pair $(\Phi_1,\Phi_2)$ satisfies the sequence condition \eqref{seq}. If,  for each $f\in L^{\Phi_1}(K)$ and $g\in L^{\Phi_2}(K)$,  $(f\ast g)(x)$ exists for almost every $x\in K$, then the set of aperiodic elements of $Z(K)$ with respect to the action $Z(K)\curvearrowright (K,\lambda)$ is empty.
\end{theorem}
\begin{proof}
 Suppose that the pair $(\Phi_1,\Phi_2)$ of Young functions satisfies the sequence condition, and $f\ast g$ exists a.e. for all $f\in L^{\Phi_1}(K)$ and $g\in L^{\Phi_2}(K)$. 
 Let, if possible, there exists an aperiodic element $a$ in $Z(K)$ with respect to the action $Z(K)\curvearrowright (K,\lambda)$. Fix a compact symmetric neighborhood $U$ of $e$ in $K$. Then, by Definition \ref{ap}, there exists an integer  $N>0$ such that for each $n\geq N$, 
 \begin{equation}\label{rel1}
 U\cap \left(a^n\ast U\right)=\varnothing.
 \end{equation}
 Note that \eqref{rel1} also implies that $U\cap \left(a^{-n}\ast U\right)=\varnothing.$
  
  Thanks to \cite[Lemma 3.2D]{jew}, there is a compact symmetric neighborhood $V$ of $e$ in $K$ such that $V\ast V\subseteq U$. So, by \eqref{rel1}, for each distinct $m,n\geq N$ we have 
  \begin{equation}\label{rel2}
  \left(a^{-mN}\ast V\right)\cap \left(a^{-nN}\ast V\right)= \varnothing \,\, \text{and}\,\,\left(a^{mN}\ast V\ast V\right)\cap\left(a^{nN}\ast V\ast V\right)=\varnothing.
  \end{equation}
  Indeed, to prove this, let us consider an element $t \in K$ which is in $a^{-mN}\ast V$ and $a^{-nN}\ast V$. Then, there exist $u,v\in V$ such that 
  $t\in \{a^{-mN}\}\ast\{u\}$ and $t\in \{a^{-nN}\}\ast\{v\}$. Now, as $a \in Z(K),$ we have
  \begin{equation*}
  u\in \{a^{mN}\}\ast \{t\}\subseteq \{a^{mN}\}\ast\{a^{-nN}\}\ast \{v\}\subseteq a^{(m-n)N}\ast U,
  \end{equation*}
  contradicting \eqref{rel1} as $u \in U$. Therefore, $\left(a^{-mN}\ast V\right)\cap \left(a^{-nN}\ast V\right)=\varnothing$. Similarly, one can see that 
  $\left(a^{mN}\ast V\ast V\right)\cap\left(a^{nN}\ast V\ast V\right)=\varnothing$.
  
   Since the pair $(\Phi_1,\Phi_2)$ satisfies the sequence condition, there are two sequences $(\alpha_n)$ and $(\beta_n)$ of nonnegative numbers such that the inequalities in \eqref{seq} hold. So,  there is an integer $N'>0$ such that 
$$\sum_{n=N'}^\infty\Phi_1(\alpha_n)<\frac{1}{\lambda(V)}\quad\text{and}\quad\sum_{n=N'}^\infty\Phi_2(\beta_n)<\frac{1}{\lambda(V\ast V)}.$$
 Define
$$f:=\sum_{n=N'}^\infty\alpha_n\chi_{a^{-nN}\ast V},$$
and
$$g:=\sum_{n=N'}^\infty\beta_n\chi_{a^{nN}\ast V\ast V},$$
 where $\chi_E$ denotes the characteristic function of $E\subseteq K$.
Hence, because of \eqref{rel2} and \eqref{inv} and applying the Monotone Convergence Theorem we have 
\begin{align*}
\int_K\Phi_1(f(x))\,d\lambda(x)&=\int_{\bigcup_{n=N'}^\infty a^{-nN}\ast V}\Phi_1(f(x))\,d\lambda(x)\\
&=\sum_{n=N'}^\infty\int_{a^{-nN}\ast V}\Phi_1(f(x))\,d\lambda(x)\\
&=\sum_{n=N'}^\infty\int_{a^{-nN}\ast V}\Phi_1(\alpha_n)\,d\lambda(x)\\
&=\lambda(V)\sum_{n=N'}^\infty\Phi_1(\alpha_n)<1,
\end{align*}
 where we have used $\Phi_1(0)=0$ in the first equality. 
 In particular, $f\in L^{\Phi_1}(K)$.
Similarly,
\begin{align*}
\int_K\Phi_2(g(x))\,d\lambda(x) 
&=\lambda(V\ast V)\sum_{n=N'}^\infty\Phi_2(\beta_n)<1,
\end{align*}
and this implies that $g\in L^{\Phi_2}(K)$. On the other hand, for each $x\in V,$ using the fact that  $\lambda(V)>0$ \cite[Theorem 1.3.12]{blm}, we have
\begin{align*}
(f\ast g)(x)&=\int_K f(y)g(y^-\ast x)\,d\lambda(y)\\
&=\sum_{n=N'}^\infty\alpha_n\int_{a^{-nN}\ast V}g(y^-\ast x)\,d\lambda(y)\\
&=\sum_{n=N'}^\infty\alpha_n\int_{a^{-nN}\ast V}\beta_n\,d\lambda(y)\\
&=\lambda(V)\sum_{n=N'}^\infty\alpha_n\beta_n=\infty,
\end{align*}
 contradicting the hypothesis $f*g$ exists a.e.
\end{proof}
If $K$ is a locally compact abelian group, then the action $Z(K)$ on $K$ is same as the natural action of $K$ on itself, because in this case we have $Z(K)=K$. So, the following result holds.
\begin{corollary}
	If a locally compact abelian group $G$ has an aperiodic element, then for each pair $(\Phi_1,\Phi_2)$ of Young functions satisfying the sequence condition \eqref{seq}, there are $f\in L^{\Phi_1}(G)$ and $g\in L^{\Phi_2}(G)$ such that $$\lambda\left(\{x\in G: (f\ast g)(x) \text{ does not exist}\}\right)>0.$$
\end{corollary}

\begin{corollary}\label{cor1}
Let $G$ be a compactly generated locally compact abelian group. Then, the following are equivalent:
\begin{enumerate}
	\item $G$ is compact.
	\item There is a pair $(\Phi_1,\Phi_2)$ of Young functions satisfying the sequence condition such that for each $f\in L^{\Phi_1}(G)$ and $g\in L^{\Phi_2}(G)$, $f\ast g$ exists a.e.
	\item For each pair $(\Phi_1,\Phi_2)$ of Young functions satisfying the sequence condition and for each $f\in L^{\Phi_1}(G)$ and $g\in L^{\Phi_2}(G)$, $f\ast g$ exists a.e.
\end{enumerate}
\end{corollary}
\begin{proof}
It is enough to prove $(2)\Rightarrow (1)$. Let $\Phi_1$ and $\Phi_2$ are two Young functions such that $(\Phi_1,\Phi_2)$ satisfy the sequence condition. Since $G$ is a  compactly generated abelian group, thanks to \cite[9.26(b)]{hew}, the set of compact elements of $G$ is a compact subgroup of $G$. So, if $G$ is not compact, it has an aperiodic element, and this contradicts Theorem \ref{main}.
\end{proof}
\begin{example}
	The additive discrete group $\mathbb Z$ is a non-compact finitely generated abelian group. So, by Corollary \ref{cor1}, for each pair $(\Phi_1,\Phi_2)$ of Young functions satisfying the sequence condition \eqref{seq}, there are $f\in l^{\Phi_1}(\mathbb Z)$ and $g\in l^{\Phi_2}(\mathbb Z)$ such that $(f\ast g)(n)=\infty$ for some $n \in  \mathbb Z$.
\end{example}

Compare the following conclusion with \cite[Theorem 1.1]{abt} and Theorem \ref{111} from T.S. Quek and L.Y.H. Yap.
\begin{corollary}\label{cor2}
	Let $G$ be a compactly generated locally compact abelian group and $2<p,q<\infty$. Then, $G$ is compact if and only if $f\ast g$ exists a.e. for all $f\in L^p(G)$ and $g\in L^q(G)$.
\end{corollary}

\section{Compact Convolution Operators}
In the sequel we assume that $K$ is a locally compact hypergroup equipped with a left-invariant measure $\lambda$, and $w$ is a weight on $K,$ that is, a positive continuous function on $K$ such that for each $x,y\in K$, $w(x\ast y)\leq w(x)\,w(y)$.
Here, $\mathcal{M}_w(K)$ denotes the set of all measures $\mu\in \mathcal{M}(K)$ with $w\mu\in \mathcal{M}(K)$. For each $\mu\in \mathcal{M}_w(K)$ we set $\|\mu\|_w:=\|w\mu\|$. In a similar way, we  can also define $L^1_w(K)$ and $L^\Phi_w(K),$ where $\Phi$ is a Young function. 

The goal of this section is to give some equivalent condition for a convolution operator from $L^1_w(K)$ into the weighted Orlicz space $L^\Phi_w(K)$ to be a compact operator. For this, we need the next theorem. 
\begin{theorem}\label{thm4}
	Let $K$ be a locally compact hypergroup and $g\in C_c(K)$. Assume that $\Phi$ is a Young function. Suppose that the bounded linear operators $T_g:L^1_w(K)\rightarrow L^\Phi_w(K)$ and $\tilde{T}_g:\mathcal{M}_w(K)\rightarrow L^\Phi_w(K)$ are defined by
	$$T_g(f):=f\ast g,\quad(f\in L^1_w(K))$$
	and 
	\begin{equation}\label{tild}
	\tilde{T}_g(\mu):=\mu\ast g,\quad(\mu\in \mathcal{M}_w(K)).
	\end{equation}
	Then, $T_g$ is compact if and only if $\tilde{T}_g$ is compact. 
\end{theorem}
\begin{proof}
	First suppose that  $T_g$ is a  compact operator. By \cite[Theorem 4.1]{ozt2}, there is a bounded left approximate identity $\{e_\alpha\}_{\alpha\in I}$ in $L^1_w(K)$ such that for each $h\in C_c(K)$, $e_\alpha\ast h\rightarrow h$ in $L^\Phi_w(K)$.
	For each $\mu\in \mathcal{M}_w(K)$ we have 
	$$\left\|\tilde{T}_g(\mu)-T_g(\mu\ast e_\alpha)\right\|_{\Phi,w}=\|\mu\ast g-\mu\ast(e_\alpha\ast g)\|_{\Phi,w}\leq \|\mu\|_w\,\| g-(e_\alpha\ast g)\|_{\Phi,w}.$$ Then, we have
	\begin{equation}\label{inc}
	\left\{\tilde{T}_g(\mu):\,\|\mu\|_w\leq 1\right\}\subseteq \overline{\left\{T_g(\mu\ast e_\alpha):\, \alpha\in I, \mu\in \mathcal{M}_w(K), \|\mu\|_w\leq 1\right\}}^{\|\cdot\|_{\Phi,w}}.
	\end{equation}
	 Now, from boundedness of the set 
	$$\{\mu\ast e_\alpha:\,\alpha\in I, \mu\in \mathcal{M}_w(K), \|\mu\|_w\leq 1\}$$
	in $L^1_w(K)$, one can see that $\left\{\tilde{T}_g(\mu):\,\|\mu\|_w\leq 1\right\}$ is compact, and so the proof of this direction is complete. Proof of the converse is easy.
\end{proof}
The following result is an Orlicz-space version of the main result of \cite[Theorem 2]{ghamed}.
\begin{theorem}\label{main2}
	Let $K$ be a locally compact hypergroup. Let $(\Phi,\Psi)$ be a pair of Young functions with $\Psi\in\Delta_2$. For  $g\in C_c(K)$, define the operator 
	$T_g:L^1_w(K)\rightarrow L^\Phi_w(K)$ by 
	$$T_g(f):=f\ast g,\quad (f\in L^1_w(K)).$$
	Then, $T_g$ is compact if and only if the function $F_g$ defined by
	\begin{equation}\label{fg}
	F_g:K\rightarrow \mathbb{R},\quad F_g(x):=\frac{1}{w(x)}\|{\rm L}_xg\|_{\Phi,w}
	\end{equation}
	for all $x\in K$, vanishes at infinity.
\end{theorem}
\begin{proof} Let, if possible, there be a $g \in C_c(K)$ such that $T_g$ is a compact operator but $F_g$ does not vanish at infinity.  Then, there is a number $\varepsilon>0$ such that for each compact set $F\subseteq K$, there exists an element $x_F\in K\setminus F$ such that
	\begin{equation}\label{111}
	\left\|\tilde T_g\left(\frac{1}{w(x_F)}\delta_{x_F}\right)\right\|_{\Phi,w}=\frac{1}{w(x_F)}\left\|L_{x_F}g\right\|_{\Phi,w}>\varepsilon,
	\end{equation}
	where $\tilde T_g$ is the operator defined by \eqref{tild}.
	By Theorem \ref{thm4}, the operator $\tilde T_g$ is also compact. Then, by boundedness of the set 
	$$\left\{\frac{1}{w(x_F)}\delta_{x_F}:\,F\subseteq K\text{ is compact}\right\}$$
	in $\mathcal{M}_w(K)$, there exists a subnet $\{x_{F_i}\}$ of $\{x_F\}$ and a function $h\in L^\Phi_w(K)$ such that
	\begin{equation}\label{lim}
	\lim_i \tilde{T}_g\left(\frac{1}{w(x_{F_i})}\delta_{x_{F_i}}\right)=h
	\end{equation}
	in $L^\Phi_w(K)$. By \eqref{111}, we have $\|h\|_{\Phi,w}\geq \epsilon$.
	So, since 
	$$\|h\|_{\Phi,w}=\sup\left\{|\langle h,f\rangle|:\, f\in L^\Psi_{w^{-1}}(K), \|f\|_{\Psi,w^{-1}}=1\right\},$$
	there is a function $\eta\in L^\Psi_{w^{-1}}(K)$ with $\|\eta\|_{\Psi,w^{-1}}=1$ such that
	$|\langle h,\eta\rangle|> \frac{\varepsilon}{2}$.
	\par 	Since $C_c(K)$ is dense in $L^\Psi_{w^{-1}}(K)$ (note that $\Psi\in\Delta_2$), there is a function $\psi\in C_c(K)$ such that  $\|\psi\|_{\Psi,w^{-1}}<\frac{3}{2}$ and 
	$$|\langle h,\psi\rangle|> \frac{\varepsilon}{2}.$$
	
	So, thanks to \eqref{lim}, there exists an index $i_0$ such that for each index $i$, if $F_{i_0}\subseteq F_i$, then
	\begin{equation}
	\left|\left\langle \tilde{T}_g\left(\frac{1}{w(x_{F_i})}\delta_{x_{F_i}}\right),\psi\right\rangle\right|> \frac{\varepsilon}{2}.
	\end{equation}
	
	
	Put $A_0:={\rm supp}(\psi)$ and $A_1:={\rm supp}(g)$. For some index $i$ we have 
	$$F_{i_0}\cup (A_0\ast A_1^-)\subseteq F_i,$$
	and so,
	$$\left\langle \tilde{T}_g\left(\frac{1}{w(x_{F_i})}\delta_{x_{F_i}}\right),\psi\right\rangle=\frac{1}{w(x_{F_i})}\,\int_{A_1}g(t)\psi(x_{F_i}t)\,dt=0,$$
	a contradiction.
	
	Conversely, let us assume that $0 \neq g \in C_c(K)$ and $F_g \in C_0(K),$ the space of all continuous functions on $K$ vanishing at infinity. Then mappings
	$$Q_1:L^\Psi(K) \rightarrow L^\Psi_{w^{-1}},\,\,\,Q_1(f)=fw,\,\,\,(f \in L^\Psi(K))$$
	and 
	$$Q_2:C_0^w(K) \rightarrow C_0(K),\,\,\,Q_2(f)=\frac{f}{w},\,\,\,(f \in C_0^w(K))$$ are isometric isomorphisms. Also, note that the operator $\tilde{T}_g$ is the adjoint of the operator
	$$Q_3:L^\Psi_{w^{-1}}(K) \rightarrow C_0^w(K),\,\,\,Q_3(f)=\langle g, L_{x^-}f \rangle,\,\,\,(f \in L^\Psi_{w^{-1}}).$$
	Now, as an application of Schauder's Theorem \cite[Chapter IV]{Convey} and having Theorem  \ref{thm4} in mind, it is enough to show that the operator
	$$Q_g:L^\Psi(K) \rightarrow C_0(K),\,\,\,Q_g=Q_2Q_3Q_1$$ is compact.
	
	To show this, let $\{f_n\}$ be a sequence $L^\Psi(K).$ For each $n \in \mathbb{N},$ let
	$$G_n:=\overline{\Big\{x\in K: |F_g(x)| \geq \frac{1}{n} \Big\}}.$$
	Then, for each $n,$ we have $G_n \subset G_{n+1}$ and since $F_g$ vanishes at infinite, $G_n$'s are compact subset of $K$. Also, for each $n \in \mathbb{N}$ and $x \in K \backslash G_n$
	\begin{align*}
	|Q_g(f_n)(x)|&=\frac{1}{w(x)} |\langle g, L_{x^-}(wf_n) \rangle|= \frac{1}{w(x)} |\langle wL_{x}g, f_n \rangle| \\&= \frac{2}{w(x)} \|wL_x g\|_{\Phi} \|f_n\|_{\Psi}=2 |F_g(x)| \|f_n\|_{\Psi} \leq \frac{2}{n} \sup_{m} \|f_m\|_{\Psi}. 
	\end{align*}
	Now, similar to the proof of second part of \cite[Theorem 3]{tab3} (also see \cite{ghamed}), by the diagonal method there is a subsequence of $\{Q_g(f_n)\}$ which converges in $C_0(K)$ and this completes the proof.   
	
\end{proof}

{\it Acknowledgments.}
The authors are grateful to the referee for making useful suggestions which have greatly improved the exposition.
The authors would like to thank  Prof. Kenneth A. Ross and Prof. Ajit Iqbal Singh for their suggestions and comments. 
Vishvesh Kumar is supported by FWO Odysseus 1 grant G.0H94.18N:
Analysis and Partial Differential Equations and by the Methusalem programme of the Ghent University Special Research Fund (BOF)
(Grant number 01M01021).

\bibliographystyle{amsplain}

\end{document}